\documentclass[11pt,makeidx, reqno]{amsart}

\usepackage{amsthm,amsfonts,amssymb,amsmath,amscd,amsrefs,thmtools}
\usepackage[all]{xy}
\usepackage{xr-hyper}
\usepackage{color}
\usepackage{bbm}
\definecolor{crimsonglory}{rgb}{0.75, 0.0, 0.2}
\definecolor{darkblue}{rgb}{0.0, 0.0, 0.55}
	\definecolor{deepskyblue}{rgb}{0.0, 0.75, 1.0}

\usepackage[colorlinks=true,linkcolor = blue, citecolor = deepskyblue]{hyperref}

\usepackage[OT2,OT1]{fontenc}
\usepackage{mathrsfs}
\usepackage{multirow}
\usepackage{graphicx}

\usepackage{xcolor}
\definecolor{bg}{RGB}{45,34,25}      
\definecolor{txt}{RGB}{230,210,180} 

\newtheoremstyle{mystyle}
  {8pt}
  {3pt}
  {\em}
  {}
  {\scshape}
  {.}
  {3pt}
  {}

\newcommand\ddfrac[2]{\frac{\displaystyle #1}{\displaystyle #2}}

\theoremstyle{mystyle}

\newcounter{exe}
  
\newtheorem{thm}{\textsc{Theorem}}

\newtheorem{lemma}[exe]{\textsc{Lemma}}
\newtheorem{prop}[exe]{\textsc{Proposition}}

\theoremstyle{definition}

\usepackage{color}

\title[]{Extreme values of class group $L$-functions}
\author{Jo\~ao Campos-Vargas}
\email{joaoccv@stanford.edu}
\address{Department of Mathematics, Stanford University, Stanford, CA 94305.}
\date{\today}

\begin{document}

\begin{abstract} We give a lower bound for the maximum value of class group $L$-functions attached to $\mathbb{Q}(\sqrt{-D})$ at the central point and show that this value is on average at least \[ \exp\left(\delta\sqrt{\frac{\log D \log \log \log D}{\log \log D}} \right)
\]
for any $\delta < \frac{1}{4}$.
\end{abstract}

\maketitle

\section{Introduction}\label{intro}

\par The goal of this paper is to exhibit large values of $L(1/2, \chi)$ where $\chi$ ranges over class group characters of the imaginary quadratic field $\mathbb{Q}(\sqrt{-D})$. The Lindel\"{o}f hypothesis predicts that $L(1/2, \chi) \ll_\epsilon D^\epsilon$, and a more quantitative result proved by Chandee and Soundararajan \cite{chandeeBounding} states that under GRH
\[L(1/2, \chi) \le \exp\left(\left(\frac{\log 2}{2}+o(1)\right) \frac{\log D}{\log \log D}\right).\]
Yet, based on probabilistic considerations, Farmer, Gonek, and Hughes \cite{farmer2007maximum} conjectured that the true order of magnitude is
\begin{equation} \label{farmerconj}
    \max_{\chi} |L(1/2, \chi)| = \exp{((c + o(1))\sqrt{\log D \log \log D})}
\end{equation}
for some positive constant $c$, and this remains the most popular belief.

\par Building on the resonance method of Soundararajan \cite{soundararajan2007extreme}, Bondarenko and Seip \cite{bondarenko2017large, bondarenkoExtreme} showed that
\[\max_{t\in [1, T]} |\zeta(1/2+it)| \ge \exp\left((1+o(1))\sqrt{\frac{\log T \log \log \log T}{\log \log T}} \right).\]
The constant appearing in their exponent was optimized by de La Bret\`{e}che and Tenenbaum \cite{de2018sommes}, who further obtained the same order of magnitude in the family of Dirichlet $L$-functions of a given modulus. Under GRH, Darbar and Maiti \cite{darbar2024large} also obtained this order of magnitude for quadratic Dirichlet $L$-functions.

\par The main result of this paper shows that this order of magnitude holds on average for class group $L$-functions over imaginary quadratic fields. This seems to be the first instance where this bound is proved unconditionally in a sympletic family, save for a function field example of Darbar and Maiti \cite{darbar2023large}.
\newpage

\begin{thm}\label{main} Let $\delta < \frac{1}{4}$. Given a fundamental discriminant $-D$, let $M_D$ be the maximum value of $L(1/2, \chi)$ where $\chi$ ranges over the characters of the class group of the imaginary quadratic field $\mathbb{Q}(\sqrt{-D})$. Let $N_X$ denote the number of fundamental discriminants $-D$ with $D \in [X, 2X]$. Then for $X$ sufficiently large
\[\Bigg({\sideset{}{^\flat} \prod_{X \le D \le 2X} M_D}\Bigg)^{\frac{1}{N_X}} \ge \exp\Bigg(\delta \sqrt{\frac{\log X \log \log \log X}{\log \log X}} \Bigg),\]
where $\displaystyle \sideset{}{^\flat} \prod_{X \le D \le 2X}$ indicates a product over fundamental discriminants $-2X \le -D \le -X$.
\end{thm}

\par The averaging feature is needed to ensure that sufficiently many small primes split in $\mathbb{Q}(\sqrt{-D})$, which we expect but do not know unconditionally. The method requires primes of size $(\log D)^{1+\epsilon}$, and even GRH does not seem sufficient to produce enough split primes of this size for every discriminant. However, GRH does yield
\[\max_{\chi} L(1/2, \chi) \ge \exp\Bigg((c+o(1)) \sqrt{\frac{\log D}{\log \log D}}\Bigg)\]
as $\chi$ ranges over the class group characters of $\mathbb{Q}(\sqrt{-D})$ for some positive $c$. This is obtained by adapting the parameters in \cite{soundararajan2007extreme}, which only require primes greater than $\log^2 D$.

\par Instead of proving this result, we describe the role of small split primes in the following more precise technical version of our main theorem.  (In what  follows, $\log_k$ denotes \hbox{the $k$-fold} iterated logarithm.) 

\begin{thm}\label{secondmain}
    Given a fundamental discriminant $-D$, let $h_D$ be the class number of $\mathbb{Q}(\sqrt{-D})$ and $\mathcal{C}$ be its class group. For $M \le h_D/3 D^\frac{1}{4}\log D$, let $\mathbb{P}$ be the set of prime ideals in $\mathbb{Q}(\sqrt{-D})$ above primes in $(e \log M \log_2 M,  \exp(\lfloor(\log_2 M)^\frac{1}{3}\rfloor) \log M \log_2M]$. Then
    \[\max_{\chi \in \widehat{\mathcal{C}}} L(1/2, \chi) \ge \exp \Bigg((1 + o(1))\sqrt{\frac{\log M \log_2 M}{\log_3 M}} \sum_{\mathfrak{p} \in \mathbb{P}} \frac{1}{\sqrt{N\mathfrak{p}}} \frac{1}{\sqrt{p}(\log p - \log_2 M - \log_3 M)}\Bigg),\]
    where $p$ denotes the prime below $\mathfrak{p} \in \mathbb{P}$.
\end{thm}

\par The magnitude of this bound depends on the splitting of primes below $\mathbb{P}$ in $\mathbb{Q}(\sqrt{-D})$. Probabilistically, half of these primes should split, yielding a bound of magnitude of Theorem \ref{main} for every discriminant. Nevertheless, it seems plausible even under GRH that all these primes are inert and the bound above is trivial, which raises an interesting question. Could extreme values for every fundamental discriminant as predicted by (\ref{farmerconj}) tell us more about the splitting of primes than GRH does, or can one still attain that order of magnitude only using primes greater than $\log^2 D$?

\section{Technical preliminaries}\label{tech} We start with the standard approximate functional equation:

\begin{lemma}\label{afe} Let $\chi$ be a nontrivial character of the class group of $\mathbb{Q}(\sqrt{-D})$. Then,
\[ L(1/2, \chi) = 2 \sum_{\mathfrak{a} \ne 0} \frac{\chi(\mathfrak{a})}{\sqrt{N\mathfrak{a}}}W\left(\frac{2\pi N \mathfrak{a}}{\sqrt{D}}\right),\]
where
\[W(x) = \frac{1}{\Gamma(1/2)} \int_x^\infty t^{\frac{1}{2}}e^{-t} \, \frac{dt}{t} \text{ for } x \ge 0\]
is a positive, decreasing function  with $W(0) = 1$ and $W(x) \le \frac{1}{e^x}$ for $x \ge 1$.
\end{lemma}

\begin{proof}
    Consider the completed $L$-function
    \[\Lambda(u, \chi) = \left(\frac{\sqrt{D}}{2\pi}\right)^u\Gamma(u)L(u, \chi),\]
    which is entire and satisfies the functional equation $\Lambda(1/2+u, \chi) = \Lambda(1/2-u, \chi)$. We write the central value $\Lambda(1/2, \chi)$ as the residue
    \[\Lambda(1/2, \chi) = \frac{1}{2\pi i}\int_{(2)} \Lambda(1/2+u, \chi)\, \frac{du}{u} - \frac{1}{2\pi i}\int_{(-2)} \Lambda(1/2+u, \chi)\, \frac{du}{u}.\]
    Using the functional equation and writing the $L$-function as the Dirichlet series, we obtain
    \[L(1/2, \chi) = 2 \sum_{\mathfrak{a} \ne 0} \frac{\chi(\mathfrak{a})}{\sqrt{N\mathfrak{a}}}W\left(\frac{2\pi N \mathfrak{a}}{\sqrt{D}}\right),\]
    where
    \[W(x) = \frac{1}{2\pi i}\int_{(2)} \frac{\Gamma(1/2+u)}{\Gamma(1/2)} x^{-u}  \, \frac{du}{u}.\]
    Lastly, write $\Gamma(1/2+u)$ as $\int_0^\infty t^\frac{1}{2} e^{-t} \, \frac{dt}{t}$ and swap the order of integration, obtaining
    \[W(x) = \frac{1}{\Gamma(1/2)} \int_0^\infty t^\frac{1}{2} e^{-t} \left(\frac{1}{2 \pi i}\int_{(2)} \left(\frac{t}{x}\right)^u \, \frac{du}{u} \right) \, \frac{dt}{t}.\]
    The inner integral is $0$ whenever $t < x$ and $1$ whenever $t > x$, which gives us
    \[W(x) = \frac{1}{\Gamma(1/2)} \int_x^\infty t^{\frac{1}{2}}e^{-t} \, \frac{dt}{t}.\]
    From this, it is clear that $W(0) = 1$ and that $W(x)$ is decreasing. For $x \ge 1$ we have
    \[W(x) \le \int_x^\infty e^{-t} \, dt = \frac{1}{e^x}\]
    concluding the proof.   
\end{proof}

We will also need the following bound:

\begin{lemma}\label{convex} With the notation of Lemma \ref{afe},
    \[ \sum_{\mathfrak{a} \ne 0} \frac{1}{\sqrt{N\mathfrak{a}}}W\left(\frac{2\pi N \mathfrak{a}}{\sqrt{D}}\right) \le (1+o(1)) D^{\frac{1}{4}}\log D.\]
\end{lemma}

\begin{proof}
    Using Lemma \ref{afe},
\[\sum_{\mathfrak{a} \ne 0} \frac{1}{\sqrt{N\mathfrak{a}}}W\left(\frac{2\pi N \mathfrak{a}}{\sqrt{D}}\right) \le \sum_{N\mathfrak{a} \le \frac{\sqrt{D}}{2\pi}} \frac{1}{\sqrt{N\mathfrak{a}}} + \sum_{N\mathfrak{a} > \frac{\sqrt{D}}{2\pi}} \frac{1}{\sqrt{N\mathfrak{a}}} \frac{\sqrt{D}}{2\pi N\mathfrak{a}}.\]
Write $\lambda(n)$ for the number of integral ideals of norm $n$ in $\mathbb{Q}(\sqrt{-D})$. As the Dedekind $\zeta$-function of this field factors as $\zeta(s)L(s, \chi_{-D})$ where $\chi_{-D}$ is the quadratic character $\left(\frac{-D}{\bullet}\right)$, we have
\[\lambda(n) = \sum_{d | n} \chi_{-D}(d) \le d(n).\]
It remains to show that
\[\sum_{n \le \frac{\sqrt{D}}{{2\pi}}} \frac{d(n)}{\sqrt{n}} + \frac{\sqrt{D}}{2\pi} \sum_{n >\frac{\sqrt{D}}{{2\pi}}} \frac{d(n)}{n^{3/2}} \le (1+o(1)) D^\frac{1}{4}\log D,\]
and this follows by partial summation and standard estimates of the divisor function.
\end{proof}

Lastly, we will need an elementary sieve result:

\begin{lemma}\label{crivo} Let $p$ be a prime and $X \ge 1$. Then

\[ \sum_{D \asymp X} \left( \frac{-D}{p} \right) = O(p\sqrt{X}),\]
where the sum is taken over all fundamental discriminants $-D \in [-2X, -X]$.    
\end{lemma}

\begin{proof}
    Given a fundamental discriminant $-D$ with squarefree part $-d$, we have
    \[-D = \begin{cases} -d & \text{if $-d \equiv 1 \pmod{4}$,} \\
    -4d & \text{if $-d \equiv 2,3 \pmod{4}$.}
    \end{cases}
    \]
We show that the sum above over squarefree numbers $-d \equiv 1 \pmod{4}$ in $[-2X, -X]$ is in fact $O(p\sqrt{X})$. Notice by periodicity that
\[\sum_{n = -2X}^{-X} \left( \frac{n}{p}\right) \chi_4^a(n)\]
is at most $4p$, where $\chi_4$ is the non-trivial character modulo $4$ and $a = 0$ or $1$. We now sieve out values that are not squarefree. By inclusion-exclusion on primes, the sum
\[\sum_{\substack{n = -2X \\ n \text{ squarefree}}}^{-X} \left( \frac{n}{p}\right) \chi_4^a(n)\]
is less than $8p\sqrt{X}$. Adding these for $a = 0$ and $1$ gives the result. The other case is handled similarly. \end{proof}


\section{The resonance method}\label{res}

\par In this section we use our family of class group $L$-functions attached to $\mathbb{Q}(\sqrt{-D})$ to describe in detail how Bondarenko and Seip \cite{bondarenko2017large} adapted the resonance method of Soundararajan \cite{soundararajan2007extreme} to produce large values. The positivity of the coefficients of the underlying Dedekind $\zeta$-function $\sum_\mathfrak{a} (N\mathfrak{a})^{-s}$ and the positivity of the right-hand side of the orthogonality relation between characters of the class group \hbox{of $\mathbb{Q}(\sqrt{-D})$} ensure that their construction works in our family.
\subsection{Setup} Given a fundamental discriminant $-D$, let $\mathcal{C}$ be the class group of $\mathbb{Q}(\sqrt{-D})$ and $h_D$ its order. For any choice of complex numbers $R_\chi$, let
\[V = \sum_{\chi \ne \chi_0} L(1/2, \chi) |R_{\chi}|^2 \text{ and } W = \sum_{\chi \ne \chi_0} |R_\chi|^2\]
where $\chi_0$ is the trivial character. Notice that $L(1/2, \chi)$ is a real number, as can be seen by pairing up conjugate ideals in Lemma \ref{afe}. Then 
\begin{equation} \label{maxL}
\max_{\chi \in \widehat{\mathcal{C}}} L(1/2, \chi) \ge V/W\end{equation} and this ratio is maximized when $R_\chi$ is the indicator of the largest $L$-value. The resonance method consists of choosing $R_\chi$ as to produce a large value for $V/W$.

\par In order to compute a concrete value, we set
\[R_\chi = \sum_{\mathcal{A} \in \mathcal{C}} \chi(\mathcal{A})r(\mathcal{A})\]
with
\[r(\mathcal{A}) = \sqrt{\sum_{\substack{\mathfrak{a} \in \mathcal{M} \\ [\mathfrak{a}] = \mathcal{A} }} f^2(\mathfrak{a})}\]
where $f$ is a real-valued function to be specified later, along with the set $\mathcal{M}$. Notice that \begin{equation} \label{W0} W \le h_D \sum_{\mathcal{A} \in \mathcal{C}} r^2(\mathcal{A}) =: W_0\end{equation}
by adding $|R_{\chi_0}|^2$ to $W$ and using orthogonality. To find a lower bound for $V$ we use Lemma \ref{afe} to write
\begin{equation}\label{V-V0}V = V_0 - E_0\end{equation}
where
\begin{equation*} V_0 = \sum_{\chi \in \widehat{\mathcal{C}}} \Bigg( 2 \sum_{\mathfrak{a} \ne 0} \frac{\chi(\mathfrak{a})}{\sqrt{N\mathfrak{a}}}W\left(\frac{2\pi N \mathfrak{a}}{\sqrt{D}}\right) \Bigg) |R_\chi|^2 \text{ and }
    E_0 = \Bigg(2 \sum_{\mathfrak{a} \ne 0} \frac{1}{\sqrt{N\mathfrak{a}}}W\left(\frac{2\pi N \mathfrak{a}}{\sqrt{D}}\right) \Bigg) |R_{\chi_0}|^2.
\end{equation*} 
\par Notice that
\begin{gather*}
V_0 = \sum_{\chi \in \widehat{\mathcal{C}}} \Bigg( 2 \sum_{\mathfrak{a} \ne 0} \frac{\chi(\mathfrak{a})}{\sqrt{N\mathfrak{a}}}W\left(\frac{2\pi N \mathfrak{a}}{\sqrt{D}}\right) \Bigg) \Bigg( \sum_{\mathcal{A}, \mathcal{B} \in \mathcal{C}} \chi(\mathcal{A})\overline{\chi}(\mathcal{B}) r(\mathcal{A})\overline{r}(\mathcal{B}) \Bigg) \\ 
= 2h_D \sum_{\mathfrak{a} \ne 0} \frac{1}{\sqrt{N\mathfrak{a}}} W\left(\frac{2\pi N\mathfrak{a}}{\sqrt{D}}\right)\Bigg(\sum_{\substack{ \mathcal{A}, \mathcal{B} \in \mathcal{C} \\ \mathcal{A}[\mathfrak{a}] = \mathcal{B}}} r(\mathcal{A})\overline{r}(\mathcal{B}) \Bigg).
\end{gather*}
Since $r$ is real, we drop the conjugate. For two classes $\mathcal{A}, \mathcal{B} \in \mathcal{C}$ with $\mathcal{A} [\mathfrak{a}] = \mathcal{B}$, 
\[r(\mathcal{A})r(\mathcal{B}) \ge  \sum_{\substack{ \mathfrak{m}, \mathfrak{n} \in \mathcal{M}, \text{ }\mathfrak{m}\mathfrak{a} = \mathfrak{n} \\ [\mathfrak{m}] = \mathcal{A}, [\mathfrak{n}] = \mathcal{B}}} f(\mathfrak{m})f(\mathfrak{n}),\]
by applying Cauchy-Schwarz pairing up ideals $\mathfrak{m}, \mathfrak{n}$ with $\mathfrak{m}\mathfrak{a} = \mathfrak{n}$. Therefore 
\begin{equation*}
V_0 \ge 2h_D \sum_{\substack{\mathfrak{m}, \mathfrak{n} \in \mathcal{M} \\ \mathfrak{m} | \mathfrak{n}}} \frac{f(\mathfrak{m})f(\mathfrak{n})}{\sqrt{N(\mathfrak{n}/\mathfrak{m})}}W\left(\frac{2\pi N(\mathfrak{n}/\mathfrak{m})}{\sqrt{D}}\right).
\end{equation*}
From this we remark that $V_0 \ge 2 W(2\pi /\sqrt{D}) W_0 \ge W_0$ for $D$ large enough.
\par By Lemma \ref{afe},
\begin{equation}\label{V0}V_0 \gg h_D \sum_{\substack{\mathfrak{m}, \mathfrak{n} \in \mathcal{M}, \text{ }\mathfrak{m} | \mathfrak{n} \\ N(\mathfrak{n}/\mathfrak{m}) < \sqrt{D}}} \frac{f(\mathfrak{m})f(\mathfrak{n})}{\sqrt{N(\mathfrak{n}/\mathfrak{m})}}.\end{equation}
In order to use (\ref{V0}) as a lower bound for $V$, we need to ensure that \begin{equation}\label{TCC}
E_0 \le c V_0    
\end{equation}
for some constant $c < 1$. When ($\ref{TCC}$) is fulfilled, (\ref{maxL}), (\ref{W0}), (\ref{V-V0}), (\ref{V0}) yield
\begin{equation}\label{maximizacao} \max_{\chi \in \widehat{\mathcal{C}}}L(1/2, \chi) \gg \ddfrac{1}{\sum_{\mathfrak{m} \in \mathcal{M}} f^2(\mathfrak{m})} \sum_{\substack{\mathfrak{m}, \mathfrak{n} \in \mathcal{M}, \text{ }\mathfrak{m} | \mathfrak{n} \\ N(\mathfrak{n}/\mathfrak{m}) < \sqrt{D}}} \frac{f(\mathfrak{m})f(\mathfrak{n})}{\sqrt{N(\mathfrak{n}/\mathfrak{m})}}.\end{equation}
We refer to (\ref{TCC}) as the \begin{it}trivial character constraint\end{it}.

\par In practice, we do not expect to get any power of $D$ as a lower bound in (\ref{maximizacao}), and the orders of magnitude of $W_0$ and $V_0$ are indistinguishable when upper bounding $E_0$. For this reason and using the observation that $W_0 \le V_0$, we fulfill the inequality $E_0 \le c W_0$ instead of (\ref{TCC}). Using Lemma \ref{convex},
\[E_0 \le (2+o(1)) D^{\frac{1}{4}}\log D \cdot \Bigg(\sum_{\mathcal{A} \in \mathcal{C}} r(\mathcal{A})\Bigg)^2\]
and by Cauchy-Schwarz,
\[\Bigg(\sum_{\mathcal{A} \in \mathcal{C}} r(\mathcal{A})\Bigg)^2  \le |\{\mathcal{A} \in \mathcal{C}: \mathcal{M} \cap \mathcal{A} \ne \emptyset \}|  \cdot \sum_{\mathcal{A} \in \mathcal{C}} r^2(\mathcal{A}) .\]
The set on the right has size at most $|\mathcal{M}|$, so by (\ref{W0}), as long as
\begin{equation}\label{sizebound} |\mathcal{M}| \le \frac{h_D}{3 D^{\frac{1}{4}}\log D}
\end{equation}
the trivial character constraint (\ref{TCC}) is fulfilled and (\ref{maximizacao}) holds.

\par Before proceeding, we take a moment to note the constraints of this method. As it will be clear later, our maximized value is a function of $|\mathcal{M}|$, and going beyond powers of $D$ for the size of this set would produce better results. However, to accomplish this we need another way of fulfilling the trivial character constraint (\ref{TCC}).

\par One possibility would be to choose $\mathcal{M}$ supported only in a few classes $\mathcal{A}$, having in mind the Cauchy-Schwarz inequality above. In this case, however, $f$ could \mbox{hardly be multiplicative} -- an important feature of the method. The only exception to this is when $\mathcal{M}$ is supported in a small subgroup of $\mathcal{C}$, but then equidistribution results seem to work against our goal. Another possibility would be to avoid Cauchy-Schwarz and instead choose a function $r$ for which we have control over the sums involved. However, this seems impractical due to the multiplicativity of 
$f$, making the sums over a single class that define $r$ hard to control.

\par Lastly we remark that the order of magnitude attained in Theorem \ref{main} is likely optimal using the current methods. More precisely, assuming the results about GCD sums from \cite{bondarenko2014gcd} for integral ideals and the upper bound (\ref{sizebound}) on the size of $\mathcal{M}$, the maximum value of the right-hand side of (\ref{maximizacao}) is already bounded above by the intended order of magnitude. 

\subsection{Parameters} We now choose the parameters to maximize (\ref{maximizacao}). We follow the original work \cite{bondarenko2017large} closely, adapting when needed with the number field features.

Let $M$ be a positive integer and $\gamma < 1/2$ be a constant. In practice, $M$ will serve as the upper bound for the size of $\mathcal{M}$. For each integral prime $\mathfrak{p}$ in $\mathbb{Q}(\sqrt{-D})$ write $p$ for the integer prime it lies above. Now $f$ will be a multiplicative function supported on a set of squarefree ideals generated by the prime ideals $\mathfrak{p}$ above primes
\[e \log M \log_2 M < p \le \exp(\lfloor(\log_2 M)^\gamma\rfloor) \log M \log_2M\]
and for such prime ideals we set
\[ f(\mathfrak{p}) = \sqrt{\frac{\log M \log_2 M}{\log_3 M}} \cdot \frac{1}{\sqrt{p}(\log p - \log_2 M - \log_3 M)}.\]
\par Having in mind that $M$ will serve as an upper bound for $|\mathcal{M}|$, we cannot take $\mathcal{M}$ to be the set of all such squarefree numbers (this has size approximately $M^{\log 2 \exp((\log_2 M)^\gamma)})$. In order to make the size of $\mathcal{M}$ less than $M$, we require an ideal in $\mathcal{M}$ to have less than $a\log M/k^2 \log_3 M$ prime ideal factors above primes in $(e^k \log M \log_2 M, e^{k+1} \log M \log_2 M]$ for each $k = 1, 2, \dots, \lfloor (\log_2 M)^\gamma \rfloor-1$, where $a$ is a fixed real number between $2$ and $1/\gamma$. 

\par Denote by $\mathbb{P}_k$ the set of prime ideals over the interval corresponding to $k$ and denote by $\mathbb{P}$ the union of all $\mathbb{P}_k$'s. The goal of this section is to prove the following result:

\begin{prop}\label{proposicao} With the definitions above, the set $\mathcal{M}$ has size at most $M$. Moreover, when $M \le D$,
\begin{multline*}\ddfrac{1}{\sum_{\mathfrak{m} \in \mathcal{M}} f^2(\mathfrak{m})} \sum_{\substack{\mathfrak{m}, \mathfrak{n} \in \mathcal{M}, \text{ }\mathfrak{m} | \mathfrak{n} \\ N(\mathfrak{n}/\mathfrak{m}) < \sqrt{D}}} \frac{f(\mathfrak{m})f(\mathfrak{n})}{\sqrt{N(\mathfrak{n}/\mathfrak{m})}} \ge \\ \exp \Bigg((1 + o(1))\sqrt{\frac{\log M \log_2 M}{\log_3 M}} \sum_{\mathfrak{p} \in \mathbb{P}} \frac{1}{\sqrt{N\mathfrak{p}}} \frac{1}{\sqrt{p}(\log p - \log_2 M - \log_3 M)} + o(1)\Bigg).\end{multline*}
\end{prop}

\par We will need three auxiliary results to establish the inequality. The first removes the truncation that bounds the norms of ideal ratios, the second shows that $\mathcal{M}$ can be replaced by the entire support of $f$, and the third computes the expression using the entire support of $f$. We start with

\begin{lemma}\label{truncation}
    When $M \le D$, we have
    \[\ddfrac{1}{\sum_{\mathfrak{m} \in \mathcal{M}} f^2(\mathfrak{m})} \sum_{\substack{\mathfrak{m}, \mathfrak{n} \in \mathcal{M}, \text{ }\mathfrak{m} | \mathfrak{n} \\ N(\mathfrak{n}/\mathfrak{m}) > \sqrt{D}}} \frac{f(\mathfrak{m})f(\mathfrak{n})}{\sqrt{N(\mathfrak{n}/\mathfrak{m})}} = o(1).\]
\end{lemma}

\begin{proof}
    Observe that
    \[\sum_{\substack{\mathfrak{m}, \mathfrak{n} \in \mathcal{M}, \text{ }\mathfrak{m} | \mathfrak{n} \\ N(\mathfrak{n}/\mathfrak{m}) > \sqrt{D}}} \frac{f(\mathfrak{m})f(\mathfrak{n})}{N(\mathfrak{n}/\mathfrak{m})} = \sum_{\mathfrak{n} \in \mathcal{M}} f^2(\mathfrak{n}) \sum_{\substack{\mathfrak{a} | \mathfrak{n} \\ N\mathfrak{a} > \sqrt{D}}} \frac{1}{f(\mathfrak{a})\sqrt{N\mathfrak{a}}}\]
    and that
    \[\sum_{\substack{\mathfrak{a} | \mathfrak{n} \\ N\mathfrak{a} > \sqrt{D}}} \frac{1}{f(\mathfrak{a})\sqrt{N\mathfrak{a}}} < \frac{1}{D^{1/8}} \sum_{\mathfrak{a} | \mathfrak{n}} \frac{1}{f(\mathfrak{a})(N\mathfrak{a})^{1/4}} = \frac{1}{D^{1/8}} \prod_{\mathfrak{p}|\mathfrak{n}} \left(1 + \frac{1}{f(\mathfrak{p})(N\mathfrak{p})^{1/4}}\right).\]
    Notice that each factor in the last product is $1+o(1)$ and there are at most $O(\log M/ \log_3 M)$ factors because $\mathfrak{n}$ is in $\mathcal{M}$. Since $M \le D$, this yields
    \[\sum_{\substack{\mathfrak{a} | \mathfrak{n} \\ N\mathfrak{a} > \sqrt{D}}} \frac{1}{f(\mathfrak{a})\sqrt{N\mathfrak{a}}} = o(1),\]
    as desired.
\end{proof}

Next we show that $\mathcal{M}$ can be replaced by entire support of $f$. 

\begin{lemma}\label{MproSupp} We have that
\[
    \sum_{\mathfrak{m}, \mathfrak{n} \in \mathcal{M}, \text{ }\mathfrak{m} | \mathfrak{n}} \frac{f(\mathfrak{m})f(\mathfrak{n})}{\sqrt{N(\mathfrak{n}/\mathfrak{m})}}   \ge (1 - o(1))\sum_{\mathfrak{m} | \mathfrak{n}} \frac{f(\mathfrak{m})f(\mathfrak{n})}{\sqrt{N(\mathfrak{n}/\mathfrak{m})}}.   \]
\end{lemma}

\begin{proof}
    We start by noticing that 
    \[\sum_{\mathfrak{m}, \mathfrak{n} \in \mathcal{M}, \text{ }\mathfrak{m} | \mathfrak{n}} \frac{f(\mathfrak{m})f(\mathfrak{n})}{\sqrt{N(\mathfrak{n}/\mathfrak{m})}} = \sum_{\mathfrak{m} | \mathfrak{n}} \frac{f(\mathfrak{m})f(\mathfrak{n})}{\sqrt{N(\mathfrak{n}/\mathfrak{m})}} - \sum_{\mathfrak{n} \notin \mathcal{M}, \text{ }\mathfrak{m} | \mathfrak{n}} \frac{f(\mathfrak{m})f(\mathfrak{n})}{\sqrt{N(\mathfrak{n}/\mathfrak{m})}}\]
    because $\mathcal{M}$ is divisor-closed (that is, all the divisors of an ideal of $\mathcal{M}$ are also in $\mathcal{M}$). Write $\mathcal{N}_k$ for the set of ideals with at least $a\log M/k^2 \log_3 M$ prime ideal factors in $\mathbb{P}_k$ and $\mathcal{N}_{k}'$ for its subset of ideals that have no prime factors outside $\mathbb{P}_k$. Then
    \[\sum_{\mathfrak{n} \notin \mathcal{M}, \text{ }\mathfrak{m} | \mathfrak{n}} \frac{f(\mathfrak{m})f(\mathfrak{n})}{\sqrt{N(\mathfrak{n}/\mathfrak{m})}} \le \sum_{k = 1}^{\lfloor (\log_2 M)^\gamma \rfloor-1} \sum_{\mathfrak{n} \in \mathcal{N}_k, \text{ } \mathfrak{m} | \mathfrak{n}} \frac{f(\mathfrak{m})f(\mathfrak{n})}{\sqrt{N(\mathfrak{n}/\mathfrak{m})}}.\]
    \par We now write these sums as products. We start with
    \[\sum_{\mathfrak{m} | \mathfrak{n}} \frac{f(\mathfrak{m})f(\mathfrak{n})}{\sqrt{N(\mathfrak{n}/\mathfrak{m})}} = \sum_{\mathfrak{n}} \frac{f(\mathfrak{n})}{\sqrt{N\mathfrak{n}}} \sum_{\mathfrak{m} | \mathfrak{n}} f(\mathfrak{m})\sqrt{N\mathfrak{m}} = \prod_{\mathfrak{p}} \left(1 + f(\mathfrak{p})^2 + \frac{f(\mathfrak{p})}{\sqrt{N\mathfrak{p}}}\right).\]
    The inner sum over $\mathcal{N}_k$ can be factored in the same way:
    \[\sum_{\mathfrak{n} \in \mathcal{N}_k, \text{ } \mathfrak{m} | \mathfrak{n}} \frac{f(\mathfrak{m})f(\mathfrak{n})}{\sqrt{N(\mathfrak{n}/\mathfrak{m})}} = \prod_{\mathfrak{p} \notin \mathbb{P}_k} \left(1 + f(\mathfrak{p})^2 + \frac{f(\mathfrak{p})}{\sqrt{N\mathfrak{p}}}\right) \Bigg(  \sum_{\mathfrak{n} \in \mathcal{N}_k'} \frac{f(\mathfrak{n})}{\sqrt{N\mathfrak{n}}} \sum_{\mathfrak{m} | \mathfrak{n}} f(\mathfrak{m})\sqrt{N\mathfrak{m}}\Bigg).\]
    Therefore the ratio between the sum over $\mathfrak{n} \in \mathcal{N}_k$ and the unrestricted sum is
    \begin{align*}
        & \ddfrac{1}{\prod_{\mathfrak{p} \in \mathbb{P}_k} \left(1 + f(\mathfrak{p})^2 + \frac{f(\mathfrak{p})}{\sqrt{N\mathfrak{p}}}\right)}\sum_{\mathfrak{n} \in \mathcal{N}_k'} \frac{f(\mathfrak{n})}{\sqrt{N\mathfrak{n}}} \sum_{\mathfrak{m} | \mathfrak{n}} f(\mathfrak{m})\sqrt{N\mathfrak{m}} \\
        \le \text{ } & \ddfrac{1}{\prod_{\mathfrak{p} \in \mathbb{P}_k} \left(1 + f(\mathfrak{p})^2 \right)} \sum_{\mathfrak{n} \in \mathcal{N}_k'} f(\mathfrak{n})^2 \prod_{\mathfrak{p} \in \mathbb{P}_k } \left( 1 + \frac{1}{f(\mathfrak{p})\sqrt{N\mathfrak{p}}}\right).
    \end{align*}
    Since each factor in the rightmost product above is at most $1 + (k+1)\sqrt{\log_3 M/\log M \log_2 M}$ and there are $|\mathbb{P}_k| < 2e^{k+1}\log M$ such factors by the prime number theorem, we obtain that
    \[\prod_{\mathfrak{p} \in \mathbb{P}_k } \left( 1 + \frac{1}{f(\mathfrak{p})\sqrt{N\mathfrak{p}}}\right) \le \exp\left(2e^{k+1}(k+1) \sqrt{\frac{\log M \log_3 M}{\log_2 M}}\right) = \exp\left(o\left(\frac{\log M}{\log_3 M}\right)\frac{1}{k^2}\right).\]
    Next since each $\mathfrak{n}$ in $\mathcal{N}_k'$ has at least $a \log M/k^2 \log_3 M$ prime factors we get
    \begin{align*}
        \ddfrac{1}{\prod_{\mathfrak{p} \in \mathbb{P}_k} \left(1 + f(\mathfrak{p})^2 \right)} \sum_{\mathfrak{n} \in \mathcal{N}_k'} f(\mathfrak{n})^2 \le & \text{ }  \ddfrac{1}{\prod_{\mathfrak{p} \in \mathbb{P}_k} \left(1 + f(\mathfrak{p})^2 \right)} b^{\frac{-a\log M}{k^2 \log_3 M}} \prod_{\mathfrak{p} \in \mathbb{P}_k}(1 + bf(\mathfrak{p})^2) \\ \le & \text{ } b^{\frac{-a\log M}{k^2 \log_3 M}} \exp\Bigg((b-1)\sum_{\mathfrak{p}\in\mathbb{P}_k}f(\mathfrak{p})^2 \Bigg)
    \end{align*}
    for any $b > 1$. The sum in the exponential can be bounded by
    \begin{align*}\sum_{\mathfrak{p} \in \mathbb{P}_k} f(\mathfrak{p})^2 \le & \frac{\log M \log_2 M}{k^2 \log_3 M} \sum_{\mathfrak{p} \in \mathbb{P}_k} \frac{1}{p} \le (2+o(1))\frac{\log M \log_2 M}{k^2 \log_3 M} \int_{e^k \log M \log_2 M}^{e^{k+1}\log M \log_2 M} \frac{1}{x \log x} \,dx
    \end{align*}
    by the prime number theorem, and the last expression evaluates to $(2+o(1))\log M/k^2 \log_3 M$. Here the factor $2$ appears because there can be at most two prime ideals over each prime.

    \par Therefore, we have shown that
    \[\sum_{\mathfrak{n} \in \mathcal{N}_k, \text{ } \mathfrak{m} | \mathfrak{n}} \frac{f(\mathfrak{m})f(\mathfrak{n})}{\sqrt{N(\mathfrak{n}/\mathfrak{m})}} \le \sum_{\mathfrak{m} | \mathfrak{n}} \frac{f(\mathfrak{m})f(\mathfrak{n})}{\sqrt{N(\mathfrak{n}/\mathfrak{m})}} \exp \left((2(b-1) - a\log b + o(1))\frac{\log M}{k^2 \log_3 M}\right)\]
    for each $k$. We now choose $b$ close to $1$ to make our exponent negative, which can be done since $a > 2$. By adding these over all $k$ and combining with our initial observation we get
    \begin{align*}\sum_{\mathfrak{n} \notin \mathcal{M}, \text{ }\mathfrak{m} | \mathfrak{n}} \frac{f(\mathfrak{m})f(\mathfrak{n})}{\sqrt{N(\mathfrak{n}/\mathfrak{m})}} \le & \, \sum_{\mathfrak{m} | \mathfrak{n}} \frac{f(\mathfrak{m})f(\mathfrak{n})}{\sqrt{N(\mathfrak{n}/\mathfrak{m})}} \sum_{k = 1}^{\lfloor (\log_2 M)^\gamma \rfloor -1}\exp\left(-\frac{c}{k^2}\frac{\log M}{\log_3 M} \right) \\
     = & \, o\left(\sum_{\mathfrak{m} | \mathfrak{n}} \frac{f(\mathfrak{m})f(\mathfrak{n})}{\sqrt{N(\mathfrak{n}/\mathfrak{m})}}\right)\end{align*}
     and therefore
     \[\sum_{\mathfrak{m}, \mathfrak{n} \in \mathcal{M}, \text{ }\mathfrak{m} | \mathfrak{n}} \frac{f(\mathfrak{m})f(\mathfrak{n})}{\sqrt{N(\mathfrak{n}/\mathfrak{m})}}   \ge (1 - o(1))\sum_{\mathfrak{m} | \mathfrak{n}} \frac{f(\mathfrak{m})f(\mathfrak{n})}{\sqrt{N(\mathfrak{n}/\mathfrak{m})}},\]
     as desired. \end{proof}

Lastly we compute the expression over the entire support of $f$.

\begin{lemma}\label{irrestrito} We have that
    \begin{multline*}\ddfrac{1}{\sum_{\mathfrak{m}} f^2(\mathfrak{m})} \sum_{\mathfrak{m} | \mathfrak{n}} \frac{f(\mathfrak{m})f(\mathfrak{n})}{\sqrt{N(\mathfrak{n}/\mathfrak{m})}} = \\ \exp \Bigg((1+o(1)) \sqrt{\frac{\log M \log_2 M}{\log_3 M}} \sum_{\mathfrak{p} \in \mathbb{P}} \frac{1}{\sqrt{N\mathfrak{p}}} \frac{1}{\sqrt{p}(\log p - \log_2 M - \log_3 M)}\Bigg). \end{multline*}
\end{lemma}
\begin{proof}
    This follows by writing the sums as products. We have
    \begin{align*} \ddfrac{1}{\sum_{\mathfrak{m}} f^2(\mathfrak{m})} \sum_{\mathfrak{m} | \mathfrak{n}} \frac{f(\mathfrak{m})f(\mathfrak{n})}{\sqrt{N(\mathfrak{n}/\mathfrak{m})}} = & \, \ddfrac{1}{\prod_\mathfrak{p} (1+f(\mathfrak{p})^2)} \prod_\mathfrak{p}\left( 1 + f(\mathfrak{p})^2 + \frac{f(\mathfrak{p})}{\sqrt{N\mathfrak{p}}}\right) \\ = & \, \prod_\mathfrak{p} \left(1 + \frac{f(\mathfrak{p})}{\sqrt{N\mathfrak{p}}(1+f(\mathfrak{p})^2)}\right) \\ = & \, \exp\Bigg((1+o(1)) \sum_{\mathfrak{p}} \frac{f(\mathfrak{p})}{\sqrt{N\mathfrak{p}}}\Bigg) \end{align*}
    using that $f(\mathfrak{p}) < 1/\log_3 M$, which gives the desired expression.
\end{proof}

We proceed to prove Proposition \ref{proposicao}:

\begin{proof}[Proof of Proposition \ref{proposicao}] We start by showing that $\mathcal{M}$ has size at most $M$. Using that
    \[\sqrt{2 \pi} m^{m+\frac{1}{2}} e^{-m} \le m! \le e m^{m+\frac{1}{2}} e^{-m}\]
    for all positive integers $m$, we get that for all $1 \le n \le m-1$,
    \[\binom{m}{n} \le \frac{m^m}{n^n(m-n)^{n-m}}\]
    and from this we get the binomial bound
    \[\binom{m}{n} \le \left(\frac{m}{n}\right)^n \left(\frac{m}{m-n}\right)^{m-n} \le \exp(n\log(m/n) + n).\]
    \par Next notice the interval $(e^k \log M \log_2 M, e^{k+1} \log M \log_2 M]$ has less than $e^{k+1}\log M$ primes by the prime number theorem. Since each prime is below at most two prime ideals, $|\mathbb{P}_k| < 2e^{k+1}\log M$. Therefore, the number of choices for the $\mathbb{P}_k$-primes in an ideal of $\mathcal{M}$ is less than
    \[ \sum_{l < \frac{a\log M}{k^2 \log_3 M}} \binom{\lfloor 2e^{k+1} \log M \rfloor}{l} \le 2 \binom{\lfloor 2e^{k+1}\log M \rfloor }{\lfloor \frac{a\log M}{k^2 \log_3 M}\rfloor} \le \exp \left( \frac{a}{k} \frac{\log M}{\log_3 M} + O\left(\frac{\log M \log_4 M}{k^2 \log_3 M}\right)\right),\]
    using that
    \[\binom{m}{n-1} \le \frac{1}{2}\binom{m}{n} \text{ for } 3n-1 \le m\]
    in the first inequality and the binomial bound in the second. Therefore
    \begin{align*}
        |\mathcal{M}| & \le \exp\left(\sum_{k = 1}^{\lfloor (\log_2 M)^\gamma \rfloor-1} \frac{a}{k} \frac{\log M}{\log_3 M} + O\left(\frac{\log M \log_4 M}{k^2 \log_3 M}\right)\right) \\
        & \le \exp ((a\gamma + o(1))\log M) \le M.
    \end{align*}
    
    We now establish the stated inequality. Using Lemma \ref{truncation},
    \[\ddfrac{1}{\sum_{\mathfrak{m} \in \mathcal{M}} f^2(\mathfrak{m})} \sum_{\substack{\mathfrak{m}, \mathfrak{n} \in \mathcal{M}, \text{ }\mathfrak{m} | \mathfrak{n} \\ N(\mathfrak{n}/\mathfrak{m}) < \sqrt{D}}} \frac{f(\mathfrak{m})f(\mathfrak{n})}{\sqrt{N(\mathfrak{n}/\mathfrak{m})}} = \ddfrac{1}{\sum_{\mathfrak{m} \in \mathcal{M}} f^2(\mathfrak{m})} \sum_{\mathfrak{m}, \mathfrak{n} \in \mathcal{M}, \text{ }\mathfrak{m} | \mathfrak{n}} \frac{f(\mathfrak{m})f(\mathfrak{n})}{\sqrt{N(\mathfrak{n}/\mathfrak{m})}} - o(1)\]
    and by Lemma \ref{MproSupp},
    \[\ddfrac{1}{\sum_{\mathfrak{m} \in \mathcal{M}} f^2(\mathfrak{m})} \sum_{\substack{\mathfrak{m}, \mathfrak{n} \in \mathcal{M}, \text{ }\mathfrak{m} | \mathfrak{n} \\ N(\mathfrak{n}/\mathfrak{m}) < \sqrt{D}}} \frac{f(\mathfrak{m})f(\mathfrak{n})}{\sqrt{N(\mathfrak{n}/\mathfrak{m})}} \ge (1-o(1))\ddfrac{1}{\sum_{\mathfrak{m}} f^2(\mathfrak{m})} \sum_{\mathfrak{m} | \mathfrak{n}} \frac{f(\mathfrak{m})f(\mathfrak{n})}{\sqrt{N(\mathfrak{n}/\mathfrak{m})}}.\]
    The result then follows from Lemma \ref{irrestrito}.
\end{proof}


\section{Proofs of Theorems \ref{main} and \ref{secondmain}}\label{proof}

In the work of Bondarenko and Seip \cite{bondarenko2017large}, the analog of Proposition \ref{proposicao} is already enough to prove the desired lower bound. The main difference here is that the lower bound from Proposition \ref{proposicao} has a sum which runs over prime ideals instead of primes, and that only gets large when many small primes split. In fact, when all primes involved are inert, the bound in Proposition \ref{proposicao} is no better than a constant -- worse than the bound coming from the first moment calculation, $\frac{1}{2}\log D$.

\par Unfortunately, GRH alone only guarantees that primes greater than $\log^2 D$ split, and it is conceivable that all primes involved in Proposition \ref{proposicao} are inert. The delicate construction of the maximizing set $\mathcal{M}$ seems to depend crucially on primes of size $(\log D)^{1 + \epsilon}$, which means that even GRH does not give a bound of the intended order of magnitude for every discriminant.

\par In order to capture the splitting of primes of size $(\log D)^{1+\epsilon}$, we consider an average. We expect each prime $p$ to split in approximately half of fundamental discriminants between $-2X$ and $-X$, and that will contribute to the sum appearing in the exponent once we take the product over all discriminants. We proceed to the proof of Theorem \ref{main}:

\begin{proof}[Proof of Theorem \ref{main}] Choose $\gamma < 1/2$ and $0 < \epsilon < .05$ such that $\gamma(\frac{1}{2}-\epsilon) > \delta$. In Proposition \ref{proposicao}, we choose $M = c\epsilon X^{1/4 - \epsilon}$ for some absolute constant $c$. The bound for the maximal central value in (\ref{maximizacao}) holds as long as \[c\epsilon X^{1/4 - \epsilon} \le \frac{h_D}{3 D^{\frac{1}{4}}\log D}.\]
We expect this to be true by Siegel's theorem, with the defect of the bound being ineffective on $D$. To fix this issue, we use a result of Tatuzawa that gives an effective version of Siegel's theorem that holds for all fundamental discriminants $-D$ with at most one exception (see Theorem 22.8 in \cite{iwaniec2021analytic}), which implies that the inequality above is true for all $D$ sufficiently large with respect to $\epsilon$ for some absolute constant $c$, allowing for at most one exceptional value of $D$.

\par We now multiply the lower bounds given by Proposition \ref{proposicao} over all fundamental discriminants except possibly the exceptional value of $D$, for which $1$ is a trivial lower bound for the maximal central value. In the imaginary quadratic field $\mathbb{Q}(\sqrt{-D})$, the number of integral ideals of norm $p$ is $1+(\frac{-D}{p})$. Therefore
\[\sideset{}{^*}\sum_{D \asymp X} 1 + \left(\frac{-D}{p}\right)\]
is the number of integral ideals of norm $p$ in fields with non-exceptional fundamental discriminant between $-2X$ and $-X$ (recall that the summation index $D \asymp X$ indicates this range, and we use the star on the summation to indicate non-exceptional discriminants).

\par By Lemma \ref{crivo}, the sum above is $(1+o(1))N_X$ given that $p = o(\sqrt{X})$. Therefore
\[\Bigg({\sideset{}{^\flat} \prod_{X \le D \le 2X} M_D}\Bigg)^{\frac{1}{N_X}} \ge \exp \Bigg(\left(\frac{1}{2} - \epsilon + o(1)\right)\sqrt{\frac{\log X \log_3 X}{\log_2 X}} \sum_{p \in I} \frac{1}{p(\log p - \log_2 M - \log_3 M )}\Bigg).\]
By the prime number theorem, the summation above is
\begin{align*} & (1+o(1)) \int_{e \log M \log_2 M}^{\exp(\lfloor(\log_2 M)^\gamma\rfloor)\log M \log_2M} \frac{1}{x(\log x - \log_2 M - \log_3 M)} \, \frac{dx}{\log x} \\  = \,& (1+o(1)) \int_{1 + \log_2 M + \log_3 M}^{\lfloor(\log_2 M)^\gamma\rfloor + \log_2 M + \log_3 M} \frac{1}{t(t-\log_2 M - \log_3 M)} \, dt \\ =\, &  (\gamma+o(1)) \frac{\log_3 M}{\log_2 M}.\end{align*}
The result follows by noting that $\gamma(\frac{1}{2} - \epsilon) > \delta$. \end{proof}

We finish by proving Theorem \ref{secondmain}:

\begin{proof}[Proof of Theorem \ref{secondmain}] Take $\gamma = 1/3$ in the construction of $\mathbb{P}$ as in Proposition \ref{proposicao}, noting that $M \le D$. The set $\mathcal{M}$ has size at most $M$, satisfying the bound (\ref{sizebound}). It follows by (\ref{maximizacao}) that
\begin{multline*}\max_{\chi \in \widehat{\mathcal{C}}}L(1/2, \chi) \ge \\ \exp \Bigg((1 + o(1))\sqrt{\frac{\log M \log_2 M}{\log_3 M}} \sum_{\mathfrak{p} \in \mathbb{P}} \frac{1}{\sqrt{N\mathfrak{p}}} \frac{1}{\sqrt{p}(\log p - \log_2 M - \log_3 M)} + O(1)\Bigg).\end{multline*}
To remove the last error term we use the first moment over this family (see Theorem 1 in \cite{duke1995class}), which yields
\[\max_{\chi \in \widehat{\mathcal{C}}}L(1/2, \chi) \ge \exp(\log \log D + O(1)).\]
This, combined with the previous bound, completes the proof. \end{proof}

\bibliographystyle{unsrt}
\bibliography{main}{}

\end{document}